\documentclass[12pt]{article}
\usepackage{amsthm}
\usepackage{amsmath}
\usepackage{amsfonts}
\usepackage{amssymb}

\newcommand{\cll}{\mathcal{L}}

\newcommand{\caa}{\mathcal{A}}
\newcommand{\chh}{\mathcal{H}}

\newcommand{\id}{\mbox{id}}

\newcommand{\si}{\sigma}
\newcommand{\de}{\delta}

\newcommand{\ot}{\otimes}

\newcommand{\trl}{\triangleleft}
\newcommand{\trr}{\triangleright}

\newcommand{\inv}{{}^{-1}}
\newcommand{\dsi}{_{\sigma}}
\newcommand{\usi}{^{\sigma}}

\newcommand{\di}{{}_{1}}
\newcommand{\dii}{{}_{2}}
\newcommand{\diii}{{}_{3}}

\newcommand{\moo}{{}_{(0)}}
\newcommand{\moi}{{}_{(-1)}}
\newcommand{\moii}{{}_{(-2)}}
\newcommand{\moiii}{{}_{(-3)}}
\newcommand{\moiv}{{}_{(-4)}}
\newcommand{\mov}{{}_{(-5)}}

\newcommand{\li}{{}_{[1]}}
\newcommand{\lii}{{}_{[2]}}

\def\rbiprod{{\cdot\kern-.33em\triangleright\!\!\!<}}
\def\lbiprod{{>\!\!\!\triangleleft\kern-.33em\cdot\, }}

\def\lrbiprod{{\ \cdot\kern-.60em\triangleright\kern-.33em\triangleleft\kern-.33em\cdot\, }}

\def\lprod{{>\!\!\!\triangleleft\kern-.33em\ \, }}

\newcommand{\allowpagebreak}
\allowdisplaybreaks

\newtheorem{Theorem}{Theorem}[section]
\newtheorem{theorem}[Theorem]{Theorem}
\newtheorem{proposition}[Theorem]{Proposition}
\newtheorem{definition}[Theorem]{Definition}

\newtheorem{example}[Theorem]{Example}

\title{Matched Pairs of Generalized Lie Algebras and Cocycle Twists}
\author{Tao Zhang}
\date{}

\begin{document}

 \maketitle

 \setcounter{section}{0}

\begin{abstract}
We introduce the conception of  matched pairs of $(H, \beta)$-Lie
algebras, construct an $(H, \beta)$-Lie algebra through them. We
prove that the cocycle twist of a  matched pair of $(H, \beta)$-Lie
algebras can also be matched.
\end{abstract}

\section{Introduction and Preliminaries}

A generalized Lie algebra in the comodule category of a cotriangular
Hopf algebra which included Lie superalgebras and Lie color algebras
as special cases has been studied by many authors, see \cite{BFM01,
FM94} and the references therein.

On the other hand, there is a general theory of matched pairs of Lie
algebras which was introduced and studied by Majid in \cite{Ma90,
Ma95}. It says that we can construct a new Lie algebra through a
matched pair of Lie algebras.

In this note, we introduce the conception of  matched pairs of $(H,
\beta)$-Lie algebras, construct an $(H, \beta)$-Lie algebra through
them. Furthermore, we prove that the cocycle twist of a  matched
pair of $(H, \beta)$-Lie algebras can also be matched.

We now fix some notation. Let $H$ be a Hopf algebra, write the
comultiplication $\Delta: H\to H\ot H$ by $\Delta(h)=\sum h\di\ot
h\dii$. When $V$ is a left $H$-comodule with coaction $\rho: V\to
H\ot V$, we write $\rho(v)=\sum v\moi\ot v\moo$. We frequently omit
the summation sign in the following context.

A pair $(H, \beta)$ is called a cotriangular Hopf algebra, if $H$ is
a Hopf algebra,  $\beta: H\ot H \to k$ is a convolution-invertible
bilinear map satisfying for all $ h, g, l\in H$,

(CT1)\quad $\beta(h\di, g\di)g\dii h\dii=\beta(h\dii, g\dii)h\di
g\di$;

(CT2)\quad$ \beta(h, gl)=\beta(h\di, g)\beta(h\dii, l)$;

(CT3)\quad$ \beta(hg, l)=\beta(g, l\di)\beta(h, l\dii)$;

(CT4)\quad$ \beta(h\di, g\di)\beta(g\dii,
h\dii)=\varepsilon(g)\varepsilon(h)$.

\noindent A map satisfying (CT2)--(CT4) is called a skew-symmetric
bicharacter. Throughout this note, we always assume $H$ is a
cotriangular Hopf algebra that is commutative and cocommutative.

A convolution invertible map $\si: H\ot H\to k$ is called a left
cocycle if for all $h, g, l\in H$,
$$\si(h\di, g\di)\si(h\dii g\dii, l)=\si(g\di, l\di)\si(h,  g\dii l\dii),$$
and a right cocycle if
$$\si(h\di g\di, l)\si(h\dii,  g\dii)=\si(h,  g\di l\di)\si(g\dii,  l\dii).$$

\begin{definition} Let $(H, \beta)$ be a cotriangular
Hopf algebra. An $(H, \beta)$-Lie algebra is a left $H$-comodule
$\cll$ together with a Lie bracket
 $[,]: \cll \ot \cll \to \cll$ which is an $H$-comodule morphism satisfying,  for
all $a, b, c\in \cll$

(1) $\beta$-anticommutativity:   
$$ [a,b] =-\beta(a\moi, b\moi)[b\moo, a\moo],$$

(2) $\beta$-Jacobi identity: 
$$[[a, b],c]+\beta(a\moi, b\moi c\moi)[[b\moo,c\moo],a\moo]+\beta(a\moi b\moi, c\moi)[[c\moo, a\moo], b\moo]=0.$$
\end{definition}

When $H=k\mathbb{Z}_2$, $\beta(x, y)=(-1)^{xy}$, for all $x, y\in
\mathbb{Z}_2$, this is exactly Lie superalgebra. When $H=kG$, where
$G$ is an abelian group with a bicharacter $\beta: G\times G\to
k^{*}$ such that $\beta(h, g)=\beta(g, h)^{-1}$ for all $h, g\in G$,
this is exactly Lie color algebra studied in \cite{CSO06, Sch79}.

\begin{example}\rm Let $A$ be a left $H$-comodule algebra. Define $[,
]_{\beta}$ to be $[a, b]_{\beta}:=ab- \sum \beta(a\moi, b\moi)$
$b\moo a\moo$. Then $(A, [, ]_{\beta})$ is an  $(H, \beta)$-Lie
algebra and is denoted by $A_{\beta}$.
\end{example}

\begin{definition}  Let $(H, \beta)$ be a cotriangular Hopf
algebra. An $(H, \beta)$-Lie coalgebra is a left $H$-comodule $\caa$
together with a Lie cobracket $ \delta: \caa \to \caa \ot \caa$
which is an $H$-comodule morphism satisfying,

(1) $\beta$-anticocommutativity:  
$$ \delta(a) =-\beta(a\li\moi, a\lii\moi) a\lii\moo\ot a\li\moo,$$

(2) $\beta$-co-Jacobi identity: 
\begin{eqnarray*}
 a\li\li\ot a\li\lii \ot a\lii+\beta(a\li\li\moi a\li\lii\moi,  a\lii\moi)a\lii \ot a\li\li\ot a\li\lii\\
+\beta(a\li\li\moi, a\li\lii\moi a\lii\moi)a\li\lii \ot a\lii \ot a\li\li=0.
\end{eqnarray*}

\noindent where we use the notion $\de(a)=\sum a\li\ot a\lii$ for
all $a\in \caa$.
\end{definition}

\begin{example}\rm Let $C$ be a left $H$-comodule coalgebra.
Define $\de_{\beta}:C\to C\ot C$ to be 
$$\de_{\beta}(c)=\sum c\di\ot
c\dii- \beta(c\di\moi, c\dii\moi)c\dii\moo\ot c\di\moo.$$ 
Then $(C, \de_{\beta})$ is an  $(H, \beta)$-Lie coalgebra and is denoted by
$(C_{\beta}, \de_{\beta})$.
\end{example}

\begin{proposition} Let $H$ be a Hopf algebra with a
skew-symmetric bicharacter $\beta: H\ot H\to k$, and suppose $\si:
H\ot H\to k$ is a left cocycle.

 (a) Define $H_{\si}$  to be  $H$ as a coalgebra, with
multiplication defined to be 
$$h\cdot_{\si}l:=\si\inv(h\di, l\di)h\dii l\dii\si(h\diii, l\diii).$$ 
Then $H$ (with a suitable
antipode) is a Hopf algebra.

(b) Define the map $\beta\dsi: H\dsi\ot H\dsi\to k$ by, for all $h, l\in H$, 
$$\beta\dsi(h, l):=\si\inv(l\di, h\di)\beta(h\dii, l\dii)\si(h\diii, l\diii).$$ 
If $(H,\beta)$ is cotriangular, then
$(H\dsi,\beta\dsi)$ is also cotriangular.

(c)  If $A$ is a left $H$-comodule algebra, define $A\usi$ to be $A$
as a vector space and $H\dsi$-comodule, with multiplication given by: 
$$a\cdot\usi b:=\si(a\moi, b\moi)a\moo b\moo.$$
Then $A\usi$ is an $H\dsi$-comodule algebra.
\end{proposition}

\begin{definition} \label{dfnlb}  An $(H, \beta)$-Lie bialgebra
$\chh$ is a vector space equipped simultaneously with an $(H,
\beta)$-Lie algebra structure $(\chh, [,])$ and an $(H, \beta)$-Lie
coalgebra $(\chh, \delta)$ structure such that the following
compatibility condition is satisfied:

(LB):
\begin{eqnarray*}
 \delta([a, b])&=& [a, b\li]\ot b\lii+ \beta(a\moi, b\li\moi) b\li\moo\ot[a\moo, b\lii]\\
 &&+ a\li\ot [a\lii, b] +\beta(a\lii\moi, b\moi) [a\li, b\moo]\ot a\lii\moo.
\end{eqnarray*}
We denoted it by $(\chh, [,], \delta)$. \end{definition}

\section{Matched Pair of $(H, \beta)$-Lie Algebras}

Let $\caa, \chh$ be both $(H, \beta)$-Lie algebras. For $a, b\in \caa$, $h, g\in \chh$, denote maps $\trr :\chh \otimes \caa \to \caa$, $\trl :\chh \otimes \caa \to \chh$, by $ \trr (h \otimes a) = h \trr a$, $ \trl (h \otimes a) = h \trl a$. If  $\chh$ is an  $(H, \beta)$-Lie algebra and  the map $\trr:\chh\ot \caa \to \caa$ satisfying
$$[h, g]\trr a=h\trr g\trr a-\beta(h\moi, g\moi)g\moo\trr h\moo\trr a ,$$ 
then $\caa$ is called a left $\chh$-module. Note that
when considering $(H, \beta)$-Lie algebras, all action maps must be
$H$-comodule maps. Thus for  $h\in \chh, a\in \caa$, we have
$$\rho(h\trr a)=\sum h\moi a\moi\ot h\moo\trr a\moo.$$ 
If $\caa$ is an $\chh$-module Lie algebra, then 
$$h\trr [a, b]=[h\trr a,b]+\beta(h\moi, a\moi)[ a\moo, h\moo\trr b]$$
 and  if  $\caa $ is an
$\chh$-module Lie coalgebra, then 
$$\de(h\trr a)=h\trr a\li\ot a\lii+\beta(h\moi, a\li\moi)a\li\moo\ot h\moo\trr a\lii.$$

\begin{definition} Let $\caa$ and  $\chh$ be $(H,\beta)$-Lie algebras. If  $\caa$ is a left $\chh$-module, $\chh$ is
a right $\caa$-module, and the following (BB1) and (BB2) hold, then
 $(\caa, \chh)$  is called  a\emph{ matched pair of $(H,\beta)$-Lie
 algebras}.

(BB1):
$$h\trr [a, b]=[h\trr a, b]+\beta(h\moi, a\moi)[ a\moo, h\moo\trr
b]+(h\trl a)\trr b-\beta(a\moi, b\moi)(h\trl b\moo)\trr a\moo,$$

(BB2): 
$$[h, g]\trl a=[h, g\trl a]+\beta(g\moi, a\moi)[h\trl a\moo,
g\moo]+h\trl (g\trr a)-\beta(h\moi, g\moi)g\moo\trl(h\moo\trr a).$$
\end{definition}

\begin{theorem}\label{th1}
 If  $(\caa, \chh)$ is a matched pair of $(H, \beta)$-Lie algebras,
 then the double cross sum $\caa\bowtie \chh$ form an $(H, \beta)$-Lie
 algebra which equals to $\caa\oplus \chh$ as linear space, but with Lie bracket
\begin{eqnarray*} [a\oplus h, b\oplus g]
&=&([a, b]+ h\trr b-\beta(g\moi, a\moi)g\moo\trr a\moo)\\
&&\oplus ([h, g]+h\trl b-\beta(g\moi, a\moi)g\moo\trl a\moo).
\end{eqnarray*}
\end{theorem}
\begin{proof} We show that the $\beta$-Jacobi identity holds for
$\caa \bowtie \chh$. By definition, $$[h, a]=h\trr a+h\trl a, [a,
h]=-\beta(a\moi, h\moi)a\moo\trr h\moo-\beta(a\moi, h\moi)a\moo\trl
h\moo.$$  
Thus we have $$[[h, g], a]=[h, g]\trr a+[h, g]\trl a,$$
for alll $h, g\in \chh,  a\in \caa$, and for the second item of $\beta$-Jacobi identity
\begin{eqnarray*}
&&\beta(h\moi g\moi, a\moi)[[a\moo, h\moo], g\moo]\\
&=&\beta(h\moii g\moii, a\moii)\beta(a\moi, h\moi)\beta((h\moo\trr
a\moo)\moi, g\moi)\\
&& g\moo\trr (h\moo\trr a\moo)\moo-\beta(h\moii g\moii,
a\moii)\beta(a\moi, h\moi)\\
&&[h\moo\trl a\moo, g\moo]+\beta(h\moii g\moii, a\moii)\beta(h\moi,
a\moi)\\
&&\beta((h\moo\trr a\moo)\moi, g\moi)g\moo\trl(h\moo\trr a\moo)\moo
\end{eqnarray*}
The right hand side is equal to:
\begin{eqnarray*}
&&\mbox{1st of RHS}\\
&=&\beta(h\moiii g\moii, a\moiii)\beta(a\moii,
h\moii)\beta(h\moi a\moi, g\moi)g\moo\trr h\moo\trr a\moo\\
&=&\beta(h\moiii, a\moiii)\beta(g\moiii, a\moiv)\beta(a\moii, h\moii)\\
&&\times\beta(h\moi, g\moi)\beta(a\moi, g\moii) g\moo\trr h\moo\trr a\moo\\
&=&\beta(g\moiii, a\moii)\beta(h\moi, g\moi)\beta(a\moi, g\moii) g\moo\trr h\moo\trr a\moo\\
&=&\beta(g\moii, a\moii)\beta(h\moi, g\moiii)\beta(a\moi, g\moi) g\moo\trr h\moo\trr a\moo\\
&=&\beta(h\moi, g\moi)g\moo\trr h\moo\trr a
\end{eqnarray*}
where we use the fact that $\trr: \chh\ot \caa\to \caa$ is a left
$H$-comodule map in the first equality, (CT2) and (CT3) for  $\beta$
in second equality, the cocommutative of $H$ in the fourth equality
and (CT4) for $\beta$ in the fifth equality. Similarly, 
$$\mbox{3rd of RHS} =\beta(h\moi, g\moi)g\moo\trl (h\moo\trr a).$$ 
and 
\begin{eqnarray*}
\mbox{2ed of RHS} &=&-\beta(h\moii, a\moii)\beta(g\moii, a\moiii)\beta(a\moi, h\moi)[h\moo\trl a\moo, g\moo] \\
&=&-\beta(g\moi, a\moi)[h\trl a\moo, g\moo].
\end{eqnarray*}

As for the third item of $\beta$-Jacobi identity,
\begin{eqnarray*}
&&\beta(h\moi, g\moi a\moi)[[g\moo, a\moo], h\moo]\\
&=&-\beta(h\moii, g\moi a\moi)\beta((g\moo\trr a\moo)\moi,
h\moi)h\moo\trr(g\moo\trr a\moo)\moo\\
&&-\beta(h\moii, g\moi a\moi)\beta((g\moo\trr a\moo)\moi,
h\moi)h\moo\trl(g\moo\trr a\moo)\moo\\
&&+\beta(h\moi, g\moi a\moi)[g\moo\trl a\moo, h\moo]
\end{eqnarray*}
The right hand side is equal to 
\begin{eqnarray*}
&&\mbox{1st of RHS}\\
&=&-\beta(h\moii, g\moii a\moii)\beta(g\moi a\moi, h\moi)h\moo\trr(g\moo\trr a\moo)\\
&=&-\beta(h\moiv, g\moii)\beta(h\moiii, a\moii)\beta(g\moi, h\moii)\\
&&\times\beta(a\moi, h\moi)h\moo\trr(g\moo\trr a\moo)\\
&=&-\beta(h\moiv, g\moii)\beta(h\moii, a\moii)\beta(g\moi, h\moiii)\\
&&\beta(a\moi, h\moi)h\moo\trr(g\moo\trr a\moo)\\
&=&-\beta(h\moii, g\moii)\beta(g\moi, h\moi)h\moo\trr(g\moo\trr
a\moo)\\
&=&-h\trr(g\trr a)
\end{eqnarray*}
Similarly, 
$$\mbox{2ed of RHS} =-h\trl(g\trr a)$$
and 
\begin{eqnarray*}
&&\mbox{3rd of RHS}=\\
&=&-\beta(h\moii, g\moi a\moi)\beta((g\moo\trl a\moo)\moi, h\moi)[h\moo, (g\moo\trl a\moo)\moo]\\
&=&-\beta(h\moii, g\moii a\moii)\beta(g\moi a\moi, h\moi)[h\moo, g\moo\trl a\moo]\\
&=&-\beta(h\moiv, g\moii)\beta(h\moiii, a\moii)\beta(g\moi, h\moi)\\
&&\times\beta(a\moi, h\moii)[h\moo, g\moo\trl a\moo]\\
&=&-\beta(h\moii, g\moii)\beta(h\moiii, a\moii)\beta(g\moi,h\moi)\\
&&\times\beta(a\moi, h\moiv)[h\moo, g\moo\trl a\moo]\\
&=&-\beta(h\moii, a\moii)\beta(a\moi, h\moi)[h\moo, g\moo\trl a\moo]\\
&=&-[h, g\trl a].
\end{eqnarray*}
Now by (BB2) and $\caa$ is a left $\chh$-module we have that the sum
of three item equals to zero. The other cases can be checked
similarly.\end{proof}

\begin{proposition} Assume that $\caa$ and $\chh$ are  $(H,
\beta)$-Lie bialgebras,  $(\caa, \chh)$ is a matched pair of $(H,
\beta)$-Lie algebras;  $\caa$ is a left $\chh$-module Lie coalgebra;
$\chh$ is a right $\caa$-module  $(H, \beta)$-Lie coalgebra. If
$(\id_{\chh}\ot \trl)(\delta_{\chh}\ot \id_{\caa})+(\trr
\ot\id_{\caa} )( \id_{\chh}\ot\delta_{\caa})=0$, i.e.

$\begin{array}{lc} (BB3):\  & \sum h\li\ot h\lii\trr a+\sum h\trl
a\li\ot a\lii=0,
\end{array}$

\noindent then $\caa\bowtie \chh$ becomes an $(H, \beta)$-Lie
bialgebra.
\end{proposition}

\begin{proof}  The Lie algebra structure is as in theorem \ref{th1}.
The Lie cobracket is the one inherited from $\caa$ and $\chh$.
$\caa$ and $\chh$ are also $(H, \beta)$-Lie sub-bialgebras of
$\caa\bowtie \chh$. So we only check  equation (LB) on $\caa\ot
\chh$. For $h\in \chh, a\in \caa$, $\de[h, a]=\de( h\trr a)+\de
(h\trl a)$, and by the ad-action on tenor product
\begin{eqnarray*}
&& h\trr \de(a)+\de(h)\trl a\\
 &=& h\trr a\li\ot a\lii(1) +h\trl a\li\ot a\lii(2)+\beta(h\moi, a\li\moi)\\
 &&a\li\moo\ot h\moo\trr a\lii(3)+ \beta(h\moi, a\li\moi)a\li\moo\ot h\moo\trl a\lii(4)\\
 &&+h\li\ot h\lii \trl a(5) +h\li\ot h\lii \trr a(6) +\beta(h\lii\moi, a\moi)\\
&&h\li\trl a\moo\ot h\lii\moo(7)+\beta(h\lii\moi, a\moi)h\li\trr
a\moo\ot h\lii\moo(8)
\end{eqnarray*}
By (BB3), $(2)+(6)=0$,  $(4)+(8)=0$. For the remaining four terms,
$\de( h\trr a)=(1)+(3)$ and $\de (h\trl a)=(5)+(7)$.
\end{proof}

\section{Cocycle Twists of Matched Pairs of $(H, \beta)$-Lie Algebras}

The cocycle twist $\caa\usi$ of an $(H, \beta)$-Lie algebra $\caa$
was introduced in \cite{BFM01}. It is an $(H, \beta\dsi)$-Lie
 algebra with the map $[, ]\usi: \caa\ot \caa\to \caa$ given by
 $$[a, b]\usi=\si(a\moi, b\moi)[a\moo, b\moo],$$
 where $ \forall a, b\in \caa$.
If in addition,  $\caa$ is a left $\chh$-module $\trr :\chh \otimes \caa \to \caa$, then we obtain that $\caa\usi$ is a
left $\chh\usi$-module $\trr\usi :\chh \otimes \caa \to \caa$, 
$$h\trr\usi  a=\si(h\moi, a\moi)h\moo\trr a\moo.$$ 
See \cite[Propostion 4.7]{BFM01}. Similarly, $\caa$ is a right $\chh$-module $\trl :\caa \otimes \chh \to \chh$, then we obtain a right $\caa\usi$-module $\trl\usi
:\chh \otimes \caa \to \chh$ by 
$$h \trl\usi  a=\si(h\moi,a\moi)h\moo\trl a\moo.$$ 
We now prove that the cocycle twist of a
matched pair of $(H, \beta)$-Lie algebras can also matched.
\begin{theorem}
 If  $(\caa, \chh)$ is a matched pair of $(H, \beta)$-Lie algebras,
 then $(\caa\usi, \chh\usi)$ is a matched pair of $(H, \beta\dsi)$-Lie algebras.
Furthermore, their double cross sum  $\caa\usi\bowtie \chh\usi$ form an $(H, \beta\dsi)$-Lie
 algebra.
\end{theorem}

\begin{proof} Note that the bracket in $\caa\usi\bowtie
\chh\usi$ is given by 
\begin{eqnarray*}
{[a, b]}\usi&=&\si(a\moi, b\moi)[a\moo, b\moo],\\
{[h, a]}\usi&=&\si(h\moi, a\moi)h\moo\trr a\moo+\si(h\moi,a\moi)h\moo\trl a\moo,\\
{[a, h]}\usi&=&-\si(h\moii,a\moii)\beta(h\moi, a\moi)h\moo\trr a\moo-\si(h\moii,a\moii)\beta(h\moi, a\moi)h\moo\trl a\moo.
\end{eqnarray*}
We check that the matched pair conditions (BB1) and (BB2) are valid on $(\caa\usi,
\chh\usi)$. We want to obtain that
\begin{eqnarray*}
h\trr\usi [a, b]\usi&=&[h\trr\usi a, b]\usi+\beta\dsi(h\moi, a\moi)[ a\moo, h\moo\trr\usi b]\usi\\
&&+(h\trl\usi a)\trr\usi b-\beta\dsi(a\moi, b\moi)(h\trl\usi b\moo)\trr\usi a\moo.
\end{eqnarray*} In fact,
\begin{eqnarray*} h\trr\usi [a, b]\usi &=&\si(h\moi, [a\moo,
b\moo]\moi)\si(a\moi,
b\moi)h\moo\trr [a\moo, b\moo]\moo\\
&=&\si(h\moi, a\moi b\moi)\si(a\moii, b\moii)h\moo\trr [a\moo,
b\moo]
\end{eqnarray*}
and \begin{eqnarray*} [h\trr\usi a, b]\usi &=& \si(h\moi,
a\moi)\si((h\moo\trr a\moo)\moi,
b\moi)[(h\moo\trr a\moo)\moo, b\moo]\\
&=& \si(h\moii, a\moii)\si(h\moi a\moi, b\moi)[h\moo\trr a\moo,
b\moo]
\end{eqnarray*}
Similarly, we get $ (h\trl\usi a)\trr\usi b = \si(h\moii,
a\moii)\si(h\moi a\moi, b\moi)(h\moo\trl a\moo)\trr b\moo$. Also,
\begin{eqnarray*}
&&\beta\dsi(h\moi, a\moi)[ a\moo, h\moo\trr\usi b]\usi\\
&=& \si(h\moiii, a\moiii)\beta(h\moii, a\moii)\si(a\moi, h\moi) [
a\moo, h\moo\trr\usi b]\usi\\
&=& \si(h\mov, a\moiv)\beta(h\moiv, a\moiii)\si(a\moii, h\moiii)\\
&&\si(h\moii, b\moii)\si(a\moi, h\moi b\moi)[a\moo, h\moo\trr b\moo]\\
&=& \si(h\mov, a\mov)\beta(h\moiv, a\moiv)\si\inv(a\moiii, h\moiii)\\
&&\si(a\moii, h\moii)\si(a\moi h\moi,  b\moi)[a\moo, h\moo\trr b\moo]\\
&=& \si(h\moiii, a\moiii)\beta(h\moii, a\moii)\si(h\moi a\moi,  b\moi)[a\moo, h\moo\trr b\moo]\\
&=& \si(h\moiii, a\moiii)\si(h\moii a\moii,  b\moii)\beta(h\moi,
a\moi)[a\moo, h\moo\trr b\moo]
\end{eqnarray*}
Similarly, we get 
\begin{eqnarray*}
\beta\dsi(a\moi, b\moi)(h\trl\usi b\moo)\trr\usi
a\moo&=&\si(h\moiii, a\moiii)\si(h\moii a\moii, b\moii)\\
&&\times\beta(h\moi,a\moi)(h\moo\trl b\moo)\trr a\moo. 
\end{eqnarray*} 
Now by the cocycle condition
of $\si$ and (BB1) on $(\caa, \chh)$, we get the result. Similar
argument can be performed for (BB2) on $(\caa\usi, \chh\usi)$.
\end{proof}

We now give the relationship between the $(H, \beta)$-Lie algebra
$\caa\usi\bowtie \chh\usi$ and $ (\caa\bowtie \chh)\usi$ by the
following theorem, the proof can easily be seen from their
construction so we omit it.

\begin{theorem}
 If  $(\caa, \chh)$ is a matched pair of $(H, \beta)$-Lie algebras,
 then   $\caa\usi\bowtie \chh\usi\cong (\caa\bowtie \chh)\usi$.
\end{theorem}

\section*{Acknowledgements}
The author would like to thank the referee for helpful comments and
suggestions.

\vskip7pt
\footnotesize{
\noindent Tao Zhang\\
College of Mathematics and Information Science,\\
Henan Normal University, Xinxiang 453007, P. R. China;\\
 E-mail address: \texttt{{zhangtao@htu.edu.cn}}


\begin{thebibliography}{9}
\itemsep-3pt


\bibitem{BFM01} Y.~Bahturin, D.~Fischman and S.~Montgomery.
Bicharacters, Twistings, and Scheuert's Theorem of  Hopf algebras,
J. Algebra, {\bf 236}(2001), 246--276.

\bibitem{FM94} D.~Fischman and S.~Montgomery.
A Schur Double Centralizer Theorem of Cotriangular Hopf algebras and
Generalized Lie Algebras, J. Algebra, {\bf 168}(1994), 594--614.

\bibitem{CSO06} X. W.~Chen, S. D.~Silvestrov and F.~van Oystaeyen.
 Representations and Cocycle Twists  of  Color Lie Algebras,
 Alg. Repres. Theory, {\bf 9}(2006), 633--650.

\bibitem{Ma90}
S.~Majid.  Matched pairs of Lie groups associated to solutions of
the Yang-Baxter equations, Pacific J. Math.  {\bf 141} (1990),
311--332.

\bibitem{Ma95} S.~Majid.  Founditions of Quantum Groups.
Cambridge University Press, Cambridge, 1995.

\bibitem{Sch79}
M.~Scheunert. Generalized Lie algebras,  J. Math. Phys. {\bf
20}(1979), 712--720.



\end{thebibliography}
\end{document}